\documentclass[11pt,leqno]{amsart}
\usepackage{amssymb,amsfonts,amsmath,amsopn,amstext,amscd,latexsym,xy,hyperref,mathrsfs,verbatim}
\input xy
\xyoption{all}

\allowdisplaybreaks[1]

\theoremstyle{remark}
\newtheorem{para}{\bf}[section]

\newtheorem{rmk}[para]{\bf Remark}

\theoremstyle{definition}
\newtheorem{exam}[para]{\bf Example}

\newtheorem{dfn}[para]{\bf Definition}

\theoremstyle{plain}
\newtheorem{thm}[para]{\bf Theorem}
\newtheorem{lemma}[para]{\bf Lemma}

\newtheorem{cor}[para]{\bf Corollary}
\newtheorem{prop}[para]{\bf Proposition}

\newenvironment{numequation}
{\addtocounter{enumi}{1}\begin{equation}}{\end{equation}}

\newenvironment{numeqnarray}
{\addtocounter{enumi}{1}\begin{eqnarray}}{\end{eqnarray}}

\newcommand{\cF}{{\mathcal F}}

\newcommand{\cH}{{\mathcal H}}

\newcommand{\cM}{{\mathcal M}}
\newcommand{\cN}{{\mathcal N}}
\newcommand{\cO}{{\mathcal O}}

\newcommand{\bB}{{\bf B}}

\newcommand{\bG}{{\bf G}}

\newcommand{\bL}{{\bf L}}

\newcommand{\bN}{{\bf N}}

\newcommand{\bP}{{\bf P}}

\newcommand{\bT}{{\bf T}}
\newcommand{\bU}{{\bf U}}

\newcommand{\bbC}{{\mathbb C}}

\newcommand{\bbG}{{\mathbb G}}

\newcommand{\bbN}{{\mathbb N}}

\newcommand{\bbQ}{{\mathbb Q}}
\newcommand{\bbR}{{\mathbb R}}

\newcommand{\bbZ}{{\mathbb Z}}

\newcommand{\frb}{{\mathfrak b}}

\newcommand{\frd}{{\mathfrak d}}

\newcommand{\frg}{{\mathfrak g}}
\newcommand{\frh}{{\mathfrak h}}

\newcommand{\frl}{{\mathfrak l}}

\newcommand{\frn}{{\mathfrak n}}

\newcommand{\frp}{{\mathfrak p}}

\newcommand{\frt}{{\mathfrak t}}
\newcommand{\fru}{{\mathfrak u}}

\newcommand{\frx}{{\mathfrak x}}

\newcommand{\GL}{{\rm GL}}
\newcommand{\Hom}{{\rm Hom}}

\newcommand{\Ad}{{\rm Ad}}

\newcommand{\ind}{{\rm ind}}
\newcommand{\Ind}{{\rm Ind}}
\newcommand{\Lie}{{\rm Lie}}
\newcommand{\lra}{\longrightarrow}
\newcommand{\midc}{\;|\;}

\newcommand{\Qp}{{\bbQ}_p}
\newcommand{\ra}{\rightarrow}
\newcommand{\Rep}{{\rm Rep}}

\newcommand{\sub}{\subset}
\newcommand{\supp}{{\rm supp}}

\newcommand\Zp{{{\bbZ}_p}}

\newcommand{\Pf}{{\it Proof. }}
\renewcommand{\qed}{{\hfill{\space} $\Box$}}

\newcommand{\M}{\mathcal{M}}

\newcommand{\N}{\mathcal{N}}

\begin{document}

\title[On some non-principal locally  analytic representations]{On some non-principal locally  analytic representations induced by Whittaker modules}
\author{Sascha Orlik}
\address{Fakultät 4 - Mathematik und Naturwissenschaften, Bergische Universit\"at Wuppertal,
Gau{\ss}stra\ss{}e 20, D-42119 Wuppertal, Germany}
\email{orlik@uni-wuppertal.de}

\normalsize

\begin{abstract}
Let $G$ be a connected split reductive $p$-adic Lie group. This paper can be seen as a continuation of \cite{O2} and is about the construction of locally analytic $G$-representations which do not lie in the principal series.
Here we consider locally analytic representations which are induced by Whittaker modules of the attached Lie algebra.
We prove that they are ind-admissible  and   topologically irreducible in case the Whittaker module is simple. On the other hand, we show that the naive Jacquet functor of these representations vanishes for all parabolic subgroups. However,  they do not satisfy the definition of supercuspidality in the  sense of Kohlhaase.
\end{abstract}

\maketitle

\tableofcontents

\section{Introduction}

%{\bf Vollstanedigkeit aller cuspidalen moduln fuer $GL_2$? exhaust all representation sof $sl_2$}

Let $L$ be a finite extension of $\Qp$ with ring of integers $O_L$. Let $G = \bG(L)$ be the group of $L$-valued points of a  connected split  reductive algebraic group $\bG$ over $O_L.$
In \cite{O2} we have considered some locally analytic $G$-representations which do not lie in the principal series as treated in \cite{OS2}. In fact those objects are globalized from cuspidal Lie algebra modules and satisfy
the supercuspidality criterion of Kohlhaase apart from the possible lack of irreducibility. In this paper we globalize amongst other things Whittaker modules and seek to prove analogous results as in loc.cit. 
Indeed we prove  that they are ind-admissible  and 
topologically irreducible when the Whittaker module is simple.
In contrast to loc.cit. they do not satisfy in general the  homological vanishing criterion in the definition
of supercuspidality of Kohlhaase \cite{K2}. Here we exemplarily discuss the case of $\GL_2$  and show that the homology groups for the opposite unipotent subgroup do not vanish in degree 1. However, the naive (Hausdorff) Jacquet module vanishes with respect to all parabolic subgroups.

We shall give now more details on the construction and statements. Let $G_0:={\bf G}(O_L)$  be the  maximal compact open subgroup of $O_L$-valued points and fix a maximal unipotent subgroup $U\subset G.$
Set $U_0=U\cap G_0$ and denote by
$\frg=\Lie(G)$ and $\fru=\Lie(U)$ the Lie algebras of $G$ and $U$, respectively. 
 Let $K$ be a finite extension of $L$ which serves as our coefficient field. Let $\eta:\fru \to K$ be a (non-degenerate) character which is induced by some locally $L$-analytic character $\hat{\eta}:U_0 \to K^\ast.$ We consider the category $\cM_\eta$ of all finitely generated $U(\frg)$-modules $M$ such that for each element $m\in M$ the expression $(x-\eta(x))^km$ vanishes for some integer $k\geq 1.$ Then the action of $\fru$ on $M$ integrates to a locally analytic action of $U_0$. 
 
 Let $D(G_0)$  be the $L$-analytic distribution algebra of $G_0$ and denote by $D(\frg,U_0)$ the subalgebra  of $D(G_0)$ generated by $\frg$ and $U_0$.  
Further we consider smooth admissible finite length representations $V$ of $U_0.$ Then $M\otimes V'$ is a $D(\frg,U_0)$-module and we set 
$$\cF^G_{\hat{\eta}}(Z,V):=c-\Ind^G_{G_0} (D(G_0) \otimes_{D(\frg,U_0)} M \otimes V')'$$ where   $c-\Ind^G_{G_0}$ is the compact induction of locally analytic representations. Thus we have defined a functor
$$\cF^G_{\hat{\eta}}: \cM_\eta \times  \Rep^{\infty,fl}_K(U_0) \to \Rep^{la}_K(G)$$
into the category of locally analytic $G$-representations. More special, let $\chi:Z(G) \to K^\ast$ be a locally analytic character  of the centre of $G$. Then we denote by 
$$\cF^G_{\hat{\eta},\chi}(Z,V)\subset \cF^G_{\hat{\eta}}(Z,V)$$ the closed subrepresentation where   $Z(G)$ acts  via $\chi$. 

Our first result is the following theorem:

\vskip10pt
\noindent {\bf Theorem A.} {\it (i) The functors $\cF^G_{\hat{\eta}}$ and $\cF^G_{\hat{\eta},\chi}$ are functorial and exact in all arguments: contravariant in $M$ and covariant in $V$.

\vskip8pt
\noindent  (ii) For all  $M\in \cM_\eta$ and for all smooth admissible finite length representations $V$ of  $U_0$, the locally analytic representations $\cF^G_{\hat{\eta}}(M,V)$ and $\cF^G_{\hat{\eta},\chi}(Z,V)$ are ind-admissible and in particular of compact type.

(Here we recall from \cite{O2} that a locally analytic representation is ind-admissible if it is a strict inductive limits of admissible representations when restricted to a compact open subgroup.)}

Concerning irreducibility statements for some of these representations we also have to consider the centre $Z(\frg)$ of $U(\frg)$. Let $\Theta:Z(\frg) \to K$ be an algebra homomorphism and denote by $K_{\Theta,\eta}$ the $1$-dimensional module of $Z(\frg)U(\fru)$ induced by $\Theta$ and $\eta$. 
We then set $N_{\Theta,\eta}:=U(\frg)\otimes_{Z(\frg)U(\fru)} K_{\Theta,\eta}$ which is an object of $\cM_\eta$. If $\eta$ is non-degenerate then Kostant \cite{Kos2} proved that this module is simple. It is up to isomorphism the only simple object in the subcategory of $\cM_\eta$ where the centre $Z(\frg)$ acts via $\Theta.$ 

\vskip10pt
\noindent {\bf Theorem B.} {\it 
Let $\Theta$, $\eta$ and $\chi$ as above.  Let $V$ be an  irreducible smooth $U_0$-representation, then $\cF^G_{\hat{\eta},\chi}(N_{\Theta,\eta},V)$ is topologically irreducible.
}

%Concerning topologically irreducibility statements of the representations above, we have to take care on the action of the centre $Z(G)$ of $G$. For this reason we  define the following variant of the above functor.

 %We fix a locally character $\psi$ on $Z(G)$ and consider the attached one dimensional dual space $J_\psi'$. Then we set
%$$\overline{\cF}_\mu(Z,V,\psi):={\rm c-\Ind}^G_{Z(G)G_0}((D(Z(G)G_0) \otimes_{D(\frg,U_0)}  (I_\eta(Z)\otimes_K V')\otimes J_\psi')').$$

%\vskip10pt

%\noindent {\bf Theorem B.} {\it Let $Z$ be given by the character $\Theta$ and $N=N_{\Theta,\eta}$. Then if $V$ is a  irreducible smooth $U_0$-represenation, then $\cF^G_{\hat{\eta}}(Z,V)$ is topologically irreducible.}

\vskip10pt
The proof of the above theorems follows the same strategy as in the principal series case \cite{OS2} and \cite{O2}, respectively. More precisely, as for Theorem A part ii) we use a recent result of Agrawal and Strauch  \cite{AS} which applies to more general closed subgroups of $G_0.$ 
Concerning Theorem B we first treat the situation where $V$ is the trivial representation.
We prove as in \cite{O2,OS2} that the module $D(G_0) \otimes_{D(\frg,U_0)} N_{\Theta,\eta}$ is simple. Then we apply the  Mackey irreducibility criterion of Bode \cite[Theorem 4.1]{B} to prove the topologically irreducibility. If $V$ is arbitrary irreducible we mimic the proof of \cite{OS2} to deduce the statement.

As for the list of content of this paper we recall in section 2 some basic facts in locally analytic representation theory and some material developed in \cite{O2}, e.g.  the notion of ind-admissibility. Section 3 is devoted to review the necessary background on Whittaker modules over the field of complex numbers. Here we follow the paper of Mili\u{c}i\'c and Soergel \cite{MS1}. In section 4 we define our functors $\cF^G_{\hat{\eta}}$, $\cF^G_{\hat{\eta},\chi}$  and prove assertions i)-ii) from Theorem A above.
Section 5 addresses the question of supercuspidality in the sense of Kohlhaase.
Finally, the irreducibility result is shown in the last section.

\begin{comment}
 
In the second case we consider Whittaker modules for $\frg$
auch gln??in Uberschrift???

We set then
$\cF(M,V):=c-Ind^G_{G_0} (D(G_0) \otimes_{D(\frg,T_0)} M \otimes V)'$ where  $D(\frg,T_0)$ is the subalgebra of $D(G_0)$ generated by $\frg$ annd $T_0.$

We extend this construction also to more genral representation.

{\bf vom kompaktren typ????}

\vskip8pt

The main result of this paper is the following:

\vskip10pt

\noindent {\bf Theorem.} {\it (i) $\cF^G$ is functorial in both arguments: contravariant in $M$ and covariant in $V$.

\vskip8pt

\noindent (ii) $\cF^G$ is exact in both arguments.

iii) $\cF(M,V)$ id admissble.

iv) Let $n=2$. Suppose that $M$ is simple and that its  central character  is.... Then $\cF(M,V)$ is topol irreducible and  admisssible.

v) $H^i(N,\cF(M,V))0=$ for all $i\geq 0.$}

\end{comment}

\vskip10pt

{\it Notation and conventions:} We denote by $p$ a prime number and consider fields $L \sub K$ which are both finite extensions of $\Qp$. 
Let $O_L$ and $O_K$ be the rings of integers of $L$, resp. $K$, and let $|\cdot |_K$ be the absolute value on $K$ such that $|p|_K = p^{-1}$. The field $L$ is our ''base field'', whereas we consider $K$ as our ''coefficient field''. For a locally convex $K$-vector space $V$ we denote by $V'_b$ its strong dual, i.e., the $K$-vector space of continuous linear forms equipped with the strong topology of bounded convergence. Sometimes, in particular when $V$ is finite-dimensional, we simplify notation and write $V'$ instead of $V'_b$. All finite-dimensional $K$-vector spaces are equipped with the unique Hausdorff locally convex topology.

We let $\bG_0$ be a split reductive group scheme over $O_L$ and $\bT_0 \sub \bB_0 \sub \bG_0$ a maximal split torus and a Borel subgroup scheme, respectively. We denote by $\bG$, $\bB$, $\bT$ the base change of $\bG_0$, $\bB_0$ and $\bT_0$ to $L$. By $G_0 = \bG_0(O_L)$, $B_0 = \bB_0(O_L)$, etc., and $G = \bG(L)$, $B = \bB(L)$, etc., we denote the corresponding groups of $O_L$-valued points and $L$-valued points, respectively. Standard parabolic subgroups of $\bG$ (resp. $G$) are those which contain $\bB$ (resp. $B$). For each standard parabolic subgroup $\bP$ (or $P$) we let $\bL_\bP$ (or $L_P$) be the unique Levi subgroup which contains $\bT$ (resp. $T$) and ${\bf U_P}$ its unipotent radical. The opposite unipotent radical is denoted by ${\bf U^-_P}.$  Finally, Gothic letters $\frg$, $\frp$, etc., will denote the Lie algebras of $\bG$, $\bP$, etc.: $\frg = \Lie(\bG)$, $\frt = \Lie(\bT)$, $\frb = \Lie(\bB)$, $\frp = \Lie(\bP)$, $\frl_P = \Lie(\bL_\bP)$, etc.. Base change to $K$ is usually denoted by the subscript ${}_K$, for instance, $\frg_K = \frg \otimes_L K$.

\vskip8pt

We make the general convention that we denote by $U(\frg)$, $U(\frp)$, etc., the corresponding enveloping algebras, {\it after base change to $K$}, i.e., what would be usually denoted by $U(\frg) \otimes_L K$, $U(\frp) \otimes_L K$ etc.
Similarly, we use the abbreviations $D(G) = D(G,K), D(P) = D(P,K)$ etc. for the locally $L$-analytic distributions with values in $K.$

\vskip10pt

{\it Acknowledgements.} 
I am very grateful to Andreas Bode for his helpful comments.

%careful reading of this paper and for all the discussions on it. In particular for pointing out to me some mistakes in a previous version and the help to remediate them.
%I thank  Georg Linden and Tobias Schmidt for some helpful remarks.

%Furthermore, we thank  for some interesting correspondence concerning category $\cO$.

\vspace{1cm}
\section{Preliminaries on locally analytic representations}

%{\bf in Appendix???}

\setcounter{enumi}{0}
 In the first half of this chapter we recall some basic facts on locally analytic representations as introduced by Schneider and Teitelbaum \cite{ST1}. In the second half we continue with some additional background on these representations established in \cite{O2}. 

For a locally $L$-analytic group $H$, let $C^{an}(H,K)$ be the locally convex vector space of locally $L$-analytic $K$-valued functions.  The dual space $D(H) = D(H,K) = C^{an}(H,K)'$ is a topological $K$-algebra which has the structure of a Fr\'echet-Stein algebra when $H$ is compact \cite{ST2}.
More generally, if $V$ is a Hausdorff locally convex $K$-vector space, let $C^{an}(H,V)$ be the $K$-vector space consisting of locally analytic functions with values in $V$. It has the structure of a Hausdorff locally convex vector space, as well.

A locally analytic $H$-representation is a Hausdorff barrelled  locally convex $K$-vector space together with a homomorphism $\rho: H \ra \GL_K(V)$
such that the action of $H$ on $V$  is continuous and the orbit maps $\rho_v: H \rightarrow V, \; h \mapsto \rho(h)(v)$, are
elements in $C^{an}(H,V)$ for all $v \in V$. 
We denote by $\Rep_K^{la}(H)$ the category of locally analytic $H$-representations on $K$-vector spaces where the morphisms are the continuous $H$-linear maps.

We recall that a Hausdorff locally convex $K$-vector space $V$ is called of {\it compact type} if it is an inductive limit of countably many Banach spaces with injective and compact transition maps, cf. \cite{ST1}. In this case, the strong dual $V'_b$ is a nuclear Fr\'echet space, cf. \cite{S1}. We denote by $\Rep_K^{la,c}(H)$ the  full subcategory of $\Rep_K^{la}(H)$ consisting of objects of compact type.  By \cite[3.3]{ST1}, the duality functor gives an equivalence of categories

\begin{equation*}\begin{array}{ccc}\label{equivalence}

\Rep_K^{la,c}(H)

& \stackrel{\sim}{\longrightarrow} &

\left\{\begin{array}{c}
\mbox{separately continuous $D(H)$-}  \\
\mbox{modules on nuclear Fr\'echet} \\
\mbox{spaces with continuous  } \\
\mbox{$D(H)$-module maps}
\end{array} \right\}^{op}.

\end{array}
\end{equation*}

\vskip8pt

\noindent In particular, $V$ is topologically irreducible if $V'_b$ is a topologically simple $D(H,K)$-module. 

For any closed subgroup $H'$ of $H$ and any locally analytic
representation $V$ of $H'$, we denote by  
$$\Ind^H_{H'}(V) := \Big\{f \in C^{an}(H,V) \midc \forall h' \in H', \forall h \in H: f(h \cdot
h') = (h')^{-1} \cdot f(h) \; \Big\}$$
the induced locally analytic representation.
The group $H$ acts on this vector space by $(h \cdot f)(x) = f(h^{-1}x)$. If $V$ is of compact type and $H/H'$ is compact then $\Ind^H_{H'}(V)$ is of compact type again and  $\Ind^H_{H'}(V)' = D(H)\otimes_{D(H')}  V'_b$ is a nuclear Fr\'echet space.

If $H$ is second countable and $H'\subset H$ is a compact open subgroup, then we set
$$c-\Ind^H_{H'}(V):=\{f \in \Ind^H_{H'}(V) \mid f \mbox{ has compact support} \}.$$
This is an $H$-stable subspace, and for any $f\in c-\Ind^H_{H'}(V)$ there are  only finitely many elements $h_1,\ldots h_r$
such that $\supp(f) = \bigcup_{i=1}^r h_i H'.$
Since $H$ is second countable we may write $c-\Ind^H_{H'}(V)$ as a countable direct sum $\bigoplus_{g\in H/H'} h \cdot V$ and supply this space with the locally convex direct sum topology.
Then $c-\Ind ^H_{H'}(V)$ is barrelled, Hausdorff and the action is locally analytic so that we get a locally analytic $H$-representation, cf. \cite{O2}. The construction is functorial and we get a functor
$$c-\Ind^H_{H'}: \Rep_K^{la}(H') \to \Rep_K^{la}(H).$$ 
The functor $c-\Ind^H_{H'}$ is left adjoint to the restriction functor $$\Rep(H)_K^{la} \to \Rep_K^{la}(H').$$ 
The  representations $c-\Ind^H_{H'}(V)$  are of compact type by \cite[Prop 1.2. ii]{ST1} but in general not admissible. However, they are  ind-admissible \cite[Lemma A.20]{O2}:

\begin{dfn}\label{dfn_qa}
 A locally analytic $H$-representation $V$ is called ind-(strongly )admissible if there is some compact open subgroup $C\subset H$ such that $V_{|C}=\varinjlim_n V_n$ is a strict inductive limit of (strongly) admissible locally analytic $C$-representations $V_n.$
\end{dfn}

The above definition does not depend on the chosen compact open subgroup  cf.  \cite[Corollary A.14]{O2} for a proof. 
It turns out that the subspaces $V_n$ are closed subrepresentation of $V_{|C}$ for all $n\in \bbN.$ Moreover, $V$ is of compact type, cf. \cite[Lemma 2.4]{O2} and $V'_b=\varprojlim_n (V_n)'_b$ is the locally convex projective limit of the Fr\'echet spaces  ($V_n)'_b$. In particular $V'_b$ is a Fr\'echet space, cf. \cite[Proposition 2.5]{O2}.

\vspace{1cm}
\section{Whittaker Lie algebra representations}
In this section we recall the theory of Whittaker modules of a Lie algebra $\frg$ over the field of complex numbers $\bbC$ in the sense of Mili\u{c}i\'c  and Soergel \cite{MS1,MS2}. In particular we assume that $\frg$ is semi-simple. We thus replace here our $p$-adic field $L$ by $\bbC.$ Finally we denote by $Z(\frg)$ the centre of the universal enveloping algebra $U(\frg)$ of $\frg.$

Let $\frb\subset \frg$ be a Borel subalgebra and $\fru$ its nilradical. We let $\frt$ be a Cartan algebra inside $\frb.$ We denote by $\Phi \subset \frt^\ast$ the attached root system and by $\Phi^+$ resp. $\Phi^-$ its subset of positive resp. negative roots. Further we denote by $\Delta$ its set of simple roots.
Let $\rho=\frac{1}{2} \sum_{\alpha \in \Phi^+}\alpha\in \frt^\ast$. For every $\alpha \in \Phi^+$ let $x_\alpha\in \fru_\alpha$ and $y_\alpha\in \fru_{-\alpha}$ be the standard generators of the root spaces, respectively.

Following Mili\u{c}i\'c and Soergel \cite{MS2} we consider the following category. 

\begin{dfn}
 Let $\cN$ be the full subcategory of all $U(\frg)$-modules $M$ such that

i) $M$ is a finitely generated $U(\frg)$-module.

ii) $M$ is locally $U(\fru)$-finite.

iii) $M$ is locally $Z(\frg)$-finite.
\end{dfn}

The category $\cN$ is abelian and every object is of finite length \cite[Thm 2.5]{MS2}.

For a maximal ideal $\Theta\subset Z(\frg)$ we let $\cN_\Theta$ be the full subcategory of objects $M$ such that $\Theta M=0.$
Further $\cN_{\overline{\Theta}}$ be the full subcategory consisting of objects $M$ for which there is an integer $k\geq 1$  such that $\Theta^k M=0.$

For a character $\eta$ of $\fru$ we let 
$\cN_\eta$ be the subcategory of $\cN$ consisting of objects $M$ such that 
$M=\{m\in M\mid \mbox{there exists an integer $k\geq 1$ with } (x-\eta(x))^km=0 \mbox{ for all } x\in \fru\}.$
We set 
$$\cN_{\overline{\Theta},\eta}=\cN_{\overline{\Theta}} \cap \cN_\eta \mbox{ and } \cN_{\Theta,\eta}=\cN_\Theta \cap \cN_\eta.$$

The category $\cN$ behaves semi-simple:

\begin{lemma}We have 
 $\cN=\bigoplus_{\Theta\in {\rm Max} Z(\frg)} \cN_{\overline{\Theta}},\,\,\, $   $\cN=\bigoplus_{\eta \in \fru^\ast} \cN_\eta$ and $$\N=\bigoplus_{\Theta,\eta} \N_{\overline{\Theta},\eta}.$$
    \end{lemma}

\begin{proof}
 This follow from \cite[Lemma 2.1, Lemma 2.2]{MS2}. See also \cite{MS1}.
\end{proof}

\begin{exam}
 Let $\Theta$ be given by an algebra homomorphism $\psi_\Theta: Z(\frg) \to \bbC$. Then
$N_{\Theta,\eta}:=U(\frg) \otimes_{Z(\frg)U(\fru)} \bbC_{\psi_\Theta,\eta}$ is an object of $\cN_{\Theta,\eta}$ where $\bbC_{\psi_\Theta,\eta}$ is $\bbC$ with the module structure of $Z(\frg)U(\fru)$ given by $\psi_\Theta\cdot \eta.$
\end{exam}

The following theorem is due to Kostant.

\begin{thm}
 Let $\eta$ be non-degenerate, i.e.  $\eta(x_\alpha)\neq 0$ for all simple roots $\alpha\in \Delta$. Then $N_{\Theta,\eta}$ is a simple $U(\frg)$-module. 
\end{thm}

\begin{proof}
 This is \cite[Thm 3.6.1]{Kos2}.
\end{proof}

The generator of $N_{\Theta,\eta}$ is up to scalar multiplication the only Whittaker vector in $N_{\Theta,\eta}$, i.e., a generating vector $v\in N_{\Theta,\eta}$ for which $x\cdot v=\eta(x)v$ for all $x\in \fru$, cf. \cite[Thm 3.4]{Kos2}.

The following theorem of Mili\u{c}i\'c and Soergel is also useful for  intrinsic reasons and will be applied in the next section. 

\begin{thm}\label{free}
 The universal enveloping algebra $U(\frg)$ is free over $Z(\frg)\otimes U(\fru)\cong Z(\frg)U(\fru).$
\end{thm}

\begin{proof}
 This is \cite[Lemma 5.7]{MS2}.
\end{proof}

For a finite-dimensional $Z(\frg)$-representation $Z$, we consider as in loc.cit. the module $I_\eta(Z)=U(\frg)\otimes_{Z(\frg)U(\fru)} (Z\otimes  \bbC_\eta).$  Denote by $\Rep^{fd}(Z(\frg))$ the category of finite-dimensional $Z(\frg)$-modules.

\begin{thm} Let $\eta$ be non-degenerate.

 i) There is an equivalence of categories $I_\eta:\Rep^{fd}(Z(\frg))  \to \cN_\eta.$
 
 ii)  Any $M\in \cN_{\Theta,\eta}$ is isomorphic to a finite sum of objects isomorphic to $N_{\Theta,\eta}$.  
\end{thm}

\begin{proof}
 Part i) is \cite[Thm. 5.9]{MS2} whereas part ii) is  \cite[Thm. 5.6]{MS2}.
\end{proof}

For later applications it is also useful to consider the following lift of the categories $\N_\eta.$ 
We let $\cM$ be the category of all finitely generated $U(\frg)$-modules $M$ which are locally $U(\fru)$-finite. Hence in contrast to $\N$ we do not assume the locally $Z(\frg)$-finite condition.  It follows by the same reasoning \cite{H1} for the BGG-category $\cO$ that $\cM$ is an abelian, noetherian category closed under submodules, quotients and finite direct sums.

We let $\cM_\eta$ be the subcategory of $\cM$ consisting of objects $M$  such that  $M=\{m\in M\mid \mbox{there exists an integer $k\geq 1$ with } (x-\eta(x))^km=0 \mbox{ for all } x\in \fru\}.$ 
Then $\N_{\overline{\Theta},\eta}$ is a full subcategory of $\M_\eta$ for all maximal ideals  $\Theta$ of  $Z(\frg).$ By the same argument as for $\N$ there is a decomposition
$$\cM=\bigoplus_{\eta \in \fru^\ast} \cM_\eta.$$
Indeed, this follows from \cite[Ch. VII, \S 1, no. 3]{Bour}.

\begin{exam}
 For a character $\eta$ of $\fru$, let
$$M_{\eta}=U(\frg) \otimes_{U(\fru)} K_{\eta}.$$ Then $M
_\eta\in \cM_\eta$ and there is a canonical epimorphism
$M_{\eta} \to N_{\Theta,\eta}.$
 \end{exam}

\vspace{1cm}
\section{The functors $\cF^G_{\hat{\eta}}$, $\cF^G_{\hat{\eta},\chi}$ }

Now we come back to our $p$-adic situation. We replace
$\mathbb{C}$ by our coefficient field $K$ and consider all objects defined in the previous section over $K$. This is harmless as all definitions are algebraic.

Let $\eta: \fru \to K$ be a character. This extends to a homomorphism $\eta:U(\fru) \to K.$ We suppose that $\eta$ is non-degenerate (i.e, $\eta(x_\alpha)\neq 0$ for $\alpha \in \Delta$) and that it is induced (via derivation) by a locally $L$-analytic  character $\hat{\eta}:U_0 \to K^\ast$  of the unipotent subgroup $U_0=U\cap \bbG(O_L).$

%\begin{prop}
% For any $V\in \Rep^{fd}(Z(\frg))$ there is a locally {\bf L ????} analytic action of $U_0$ on $I_\eta(V)$ such that $U_0$ acts on $\im(V \to I_\eta(V)$ by $\hat{\eta}\cdot \id.$
%\end{prop}

%We denote the locally analitic representation by $I_{\hat{\eta}}$......

\begin{prop}
 For any $M\in \cM_\eta$ there is a locally $L$-analytic action of $U_0$ on $M$.
\end{prop}

\begin{proof}
 We consider the twisted module $M\otimes K_{-\eta}$ which is a locally $U(\fru)$-finite module where $\fru$ acts trivially. Here the action of the Lie algebra $\fru$ integrates uniquely an algebraic action of $\bU$ via the exponential series, i.e.,   we define $u\cdot m := \sum_{n \ge 0} \frac{\frx^n}{n!}m$ where $u = \exp(\frx)\in {\bf U}(\overline{K})$, 
By untwisting, i.e. by considering the tensor product $M\otimes K_{-\eta}\otimes K_{\hat{\eta}}= M$  we get the desired result.
\end{proof}

 Let $\Rep^{\infty,fl}(U_0)$ be the category of smooth admissible  $U_0$-representation of finite length and fix an object $V$ in here. Note that $V$ is finite-dimensional since $U_0$ is compact. Denote by $D(\frg,U_0)\subset D(G_0)$ the subalgebra generated by $\frg$ and $D(U_0).$ 
 Then we get for any $M\in \cM_\eta$ by the proposition above a locally analytic $U_0$-action on $M\otimes_K V'$  and thus via the trivial action $\frg$ on $V$ a $D(\frg,U_0)$-module structure on $M\otimes_K V'.$ 
We set $$X_{\hat{\eta}}(M,V)=D(G_0) \otimes_{D(\frg,U_0)}  M \otimes_K´ V'$$ 
which is a separately continuous $D(G_0)$-module on a nuclear Fr\'echet space. Hence its topological dual $X_{\hat{\eta}}(M,V)'$ is a locally analytic $G_0$-representation of compact type. We
put
$$\cF^G_{\hat{\eta}}(M,V):={\rm c-\Ind}^G_{G_0} X_{\hat{\eta}}(M,V)'$$
which is thus a locally analytic $G$-representation.
%and
%$$X_\eta(Z,V)=(D(G_0) \otimes_{D(U_0,\frg)}  I_{\hat{\eta}}(Z) \otimes V')'.$$
This construction is functorial in each entry. Hence for a fixed locally analytic character $\hat{\eta}$  we get thus a bifunctor 
$$\cF^G_{\hat{\eta}}: \cM_\eta \times \Rep_K^{\infty,fl} (U_0)\to \Rep^{la}_K (G).$$

\vspace{0.5cm}
\begin{prop}\label{biexact}
 The functor $\cF^G_{\hat{\eta}}$ is bi-exact.
\end{prop}

\begin{proof}
 Since ${\rm c-\Ind}^G_{G_0}$ is an exact functor  it is enough to see that the expression $X_{\hat{\eta}}(M,V)$ defines an exact functor in $M$ and $V.$  As for the smooth entry, we choose as in \cite[Prop. 4.9]{OS2} a locally analytic section $s$ of the projection $G_0 \to G_0/U_0$ and set $\cH=s(G_0/U_0).$  Then $X(M,V)= D(\cH)\otimes_{U(\frg)} M \otimes_K V'$ as vector spaces. 
 As for the exactness in $M$ this follows from the result \cite[Corollary 7.8.7]{AS} by Agrawal and Strauch treating more general closed subgroups of $G.$
\end{proof}

We can give as in \cite{OS2} another description of the above representation. We consider for an object $M$ of 
$\M_\eta$, a finite-dimensional $U(\fru)$-subspace $W$  which generates $M$ as  a $U(\frg)$-module and which is equipped with a locally analytic $U_0$-action. Let 
$\frd=\ker (U(\frg)\otimes_{U(\fru)} W  \to M)$ be the kernel of the natural morphism. Let $V$ be a smooth admissible $U_0$-representation of finite length. Then there is similarly as in \cite{OS2} the non-degenerate pairing
\begin{equation*}
\begin{array}{rccc}
\langle \cdot , \cdot \rangle: & \left(D(G_0) \otimes_{D(U_0)}  W\otimes V' \right) \otimes_K \Ind^{G_0}_{U_0}(W' \otimes V) & \lra & C^{an}(G_0,V) .\\
&&&\\
& (\delta \otimes w) \otimes f & \mapsto & \Big[ g \mapsto \big(\delta \cdot_r (f(\cdot)(w))\big)(g)\Big]
\end{array}
\end{equation*}
\noindent Here, by definition, we have $\big(\delta \cdot_r (f(\cdot)(w))\big)(g) = \delta(x \mapsto f(gx)(w))$.

Then we may identify similarly to \cite{OS2} our representation $X_{\hat{\eta}}(M,V)'$ with the closed subrepresentation $$\Ind^{G_0}_{U_0}(W'\otimes V)^\frd:=\{f\in \Ind^{G_0}_{U_0}(W'\otimes V) \mid \langle x,f \rangle=0 \, \forall x\in \frd  \}$$ of $\Ind^{G_0}_{U_0}(W' \otimes V)$ and thus $\cF^G_{\hat{\eta}}(M,V)$ with 
$${\rm c-\Ind}^G_{G_0}(\Ind^{G_0}_{U_0}(W'\otimes V)^\frd).$$

\begin{exam}\label{example_Theta}
Let $M=N_{\Theta,\eta}$ be as above. Then
$$\cF^G_{\hat{\eta}}(N_{\Theta,\eta})=c-\Ind^G_{G_0}(\Ind^{G_0}_{U_0}(K_{\hat{\eta}}')^{Z(\frg)=\Theta})$$
is the closed subspace of $c-\Ind^G_{G_0}(\Ind^{G_0}_{U_0}(K_{\hat{\eta}}'))$ where $Z(\frg)$ acts by $\psi_\Theta'.$
\end{exam}

If $V={\bf 1}^\infty$ is the trivial representation of $U_0$,  then we simply omit it from the input in our functor as in \cite{OS2}. E.g., we write
$\cF^G_{\hat{\eta}}(M)$ for  $\cF^G_{\hat{\eta}}(M,{\bf 1}^\infty).$

\begin{prop}
 Let $M\in \M_\eta$ and let $V\in \Rep^{\infty,fl}_K(U_0)$.   Then the representation $X_{\hat{\eta}}(M,V)'$ is strongly admissible and $\cF^G_{\hat{\eta}}(M,V)$ is ind-admissible.
\end{prop}

\begin{proof}
 The Lie algebra representation $M$ is a  finitely generated $U(\frg)$-module. As $V$ is finite-dimensional the module $X_{\hat{\eta}}(M,V)$ is finitely generated as a $D(G_0)$-module, as well. Thus by definition $X_{\hat{\eta}}(M,V)'$ is strongly admissible .  

 The proof of the ind-admissibility is the same as in 
 \cite{O2} which we recall for convenience.
 Let $S\subset G$ be a set of representatives for the double cosets $G_0 \backslash G/G_0.$ 
 As for representations of finite groups \cite{Se} we have  
 \begin{eqnarray*}
  c-\Ind^G_{G_0}(X_{\hat{\eta}}(M,V)')_{|G_0}  =  \bigoplus_{s \in S} \Ind^{G_0}_{G_0 \cap  sG_0s^{-1}} \big(s{X_{\hat{\eta}}(M,V)'} \big)
  \end{eqnarray*}
  Now $G$ is second countable hence $G_0\backslash G/G_0$ is countable. Since any countable locally convex sum is a strict inductive limit and $sG_0s^{-1}\cap G_0$ is of finite index in $G_0$ the claim follows.  
\end{proof}

Next we want to specialize our functor to the subcategory $\cN_\eta.$
Let $Z$ be a finite-dimensional $Z(\frg)$-representation.
Then $$I_\eta(Z)=U(\frg)\otimes_{Z(\frg)U(\fru)} Z\otimes K_\eta \in \cN_\eta$$ and we have a locally analytic $U_0$-action on $I_\eta(Z)$. Thus we get via the trivial action $\frg$ on $V$ a $D(\frg,U_0)$-module structure on $I_\eta(Z)\otimes_K V'.$ 
 We set $$\tilde{\cF}^G_{\hat{\eta}}(Z,V):=\cF^G_{\hat{\eta}}(I_\eta(Z),V).$$
This construction is functorial in each entry. Hence for a fixed $\hat{\eta}$ we get thus a bifunctor 
$$\tilde{\cF}^G_{\hat{\eta}}: \Rep_K^{fd}(Z(\frg)) \times \Rep_K^{\infty,fl} (U_0)\to Rep^{la}_K (G).$$

\vspace{0.5cm}
\begin{prop}
 The functor $\tilde{\cF}^G_{\hat{\eta}}$ is bi-exact.
\end{prop}

\begin{proof}
 This follows from Proposition \ref{biexact} and the exactness of the functor $I_\eta.$
\end{proof}

There is the also the following variant of the above functors concerning the prescribed action of the centre. 
Let $\chi:Z(G) \to K^\ast$ be a locally analytic character  of the centre of $G$. Let $\chi_0$ its restriction to $Z(G_0)=Z(G)\cap G_0.$ Then we have by duality the one-dimensional $D(Z(G_0))$-module $K_\chi'$. 
Put
$$X_{\hat{\eta},\chi}(M,V)=D(G_0) \otimes_{D(\frg,U_0Z(G_0))}  M \otimes_K´ V'\otimes_K K_\chi'$$ 
and set 
$$\cF^G_{\hat{\eta},\chi}(Z,V):=c-\Ind^G_{Z(G)G_0}(X_{\hat{\eta},\chi}(M,V))',$$
where $Z(G)$ acts on $X_{\hat{\eta},\chi}(M,V))'$  via $\chi$. Then by repeating essentially the above proofs, we get

\begin{prop}
 The functor $\cF^G_{\hat{\eta},\chi}$ is bi-exact , $X_{\hat{\eta},\chi}(M,V)$ is strongly admissible and $\cF^G_{\hat{\eta},\chi}(Z,V)$ is ind-admissible. \qed
 \end{prop}

\vspace{1cm}
\section{Cuspidality}
\setcounter{enumi}{0}

In this section we will show amongst other things that the ordinary (Hausdorff) Jacquet modules of $\cF^G_{\hat{\eta}}(N_{\Theta,\eta},V)$ with respect to  proper parabolic subgroup vanish. However, we will see that the higher homology groups do not vanish in general. We keep the notation from the previous section and fix further a smooth representation $V$ of $U_0.$ For convenience      
we set $N= N_{\Theta,\eta}$.  

For the proof of the irreducibility result in the next section we need the following notation. Let $B'\subset G$ be another Borel subgroup, $\frb'=\Lie (B')$ and 
and $\fru':=[\frb',\frb'].$  
Let $\eta':\fru' \to K$ be non-degenerate  function and indicate  by 
\begin{eqnarray*}
 \fru'_{\eta'}\times M \to M  \\
 (x,m) \mapsto xm-\eta'(x)m
\end{eqnarray*}
the $\eta'$-twisted action of $\fru'$ for any $\fru'$-module $M.$

\begin{prop}\label{H^0_2}
 i) Let $P\subset G$ be a parabolic subgroup and $U_P$ its unipotent radical. Set $U_{P,0}=G_0 \cap U_P.$ Then $H^0(U_{P,0},X_{\hat{\eta}}(N,V))=0.$

 ii) We have $H^0(\fru'_{\eta'},X_{\hat{\eta}}(N,V))=K\cdot\delta_g\cdot v^+$ if there is some $g\in G_0$ with $Ad(g)(\fru,\eta)=(\fru',\eta')$. Otherwise we have $H^0(\fru_{\eta'},X_{\hat{\eta}}(N,V))=0.$
 \end{prop}

 For the proof we need some preparations.
 Let $\epsilon:U(\frg) \to K$ be the augmentation homomorphism. Then $\widehat{U(\frg)}:=\varprojlim_n U(\frg)/I^n$ is the formal completion of $U(\frg).$ More generally for every $U(\frg)$-module $M$ its formal completion is given by $\widehat{M}:=\varprojlim_n M/I^nM.$ This is a module for $\widehat{U(\frg)}$. 
 \begin{exam}
  For $M=M_\eta=U(\frg)\otimes_{U(\fru)} K_\eta$ we have $\widehat{M_\eta}=\widehat{U(\frb^-)}\otimes K_\eta$ since by Poincar\'e Birkhoff Witt we have  $M_\eta=  U(\frb^-)\otimes K_\eta$.
 \end{exam}

  Concerning an explicit description of  the formal completion $\widehat{N}$ of $N_{\Theta,\eta}$ we let $y_1,\ldots, y_t$ be basis of $\fru^-$. Then  $\widehat{U}(\fru^-)=\prod_{\underline{n}\in \bbN_0^t} K y_1^{n_1}\cdots y_r^{n_r}.$ This in an obvious way a module for $U(\fru^-)$ and the module structure extends to a representation of $U^-$. Indeed we consider the exponential $\exp_{U^-}: \fru^- \dashrightarrow U^-$ and the inverse map $\log_{U^-}$. We  define for $u \in U^-$ and $v = \sum_{\underline{n}} v_{\underline{n}} \in \widehat{U}(\fru^-)$:

$$\delta_u \cdot v = \sum_{\underline{n}} \sum_{i \geq 0} \frac{1}{i!} \log_{U^-}(u)^i \cdot v_{\underline{n}} \,.$$

\noindent Note that this sum is well-defined in $\widehat{U}(\fru^-)$, because $\log_{U^-}(u)$ is in $\fru^-$, and there are hence only finitely many terms of a given homogeneous degree in this sum.  Moreover, it gives an action of $U^-$ on $\hat{U}(\fru^-)$ because $
\exp(\log_{U^-}(u_1)) \cdot \exp(\log_{U^-}(u_2)) 
= \exp\big(\cH(\log_{U^-}(u_1),\log_{U^-}(u_2))\big) 
 = \exp(\log_{U^-}(u_1u_2))$ where $\cH(X,Y) = \log(\exp(X)\exp(Y))$ is the Baker-Campbell-Hausdorff series (which converges on $\fru^-$, as $\fru^-$ is nilpotent). It is then immediate that this action is compatible with the action of $U(\fru^-)$. 
 
 Since $U(\frg)$ is finite free over $Z(\frg)U(\fru)$ we find by the Harish-Chandra isomorphism  finitely many elements $h_1,\ldots,h_r\in U(\frt)$ such that the elements of $U(\fru^-)\oplus U(\fru^-)h_1 \oplus\cdots \oplus U(\fru^-)h_r$ form a basis of $U(\frg)$ over $Z(\frg)U(\fru)$.  Then
\begin{numequation}\label{presentation_N}
 \hat{N} = \hat{U}(\fru^-) \times  \hat{U}(\fru^-)h_1 \times \cdots \times \hat{U}(\fru^-)h_r \otimes K_{\Theta,\eta}.
 \end{numequation}
 Here we have again an action of $U^-$ which is given component wise.

\begin{proof}[{Proof of Proposition \ref{H^0_2}}] 
 i) We follow the proof of  \cite[Thm. 6.1]{O2} and start with the observation that  $H^0(U_{P,0},X_{\hat{\eta}}(N,V))=H^0(\fru_P,X_{\hat{\eta}}(N,V))^{U_{P,0}}$. Thus it suffices to see that  $H^0(\fru_P,X_{\hat{\eta}}(N,V))=0.$ Since the action of $\fru_P$ on $V$ is trivial we may suppose that $V={\bf 1}^\infty$ is the trivial representation, cf. \cite[Lemma 3.4]{OSch}.  
   
 Let $I  \sub G_0$ be the standard Iwahori subgroup. 
 For $w\in W$, let $N^w$ be the $D(\frg, I \cap wU_0w^{-1})$-module with underlying
vector space $N$ twisted action given by composition of the given one with conjugation with $w$. Let 
 \begin{numequation}\label{Iwahori}
w^{-1} I w=(U_{0}^-
    \cap w^{-1} I w)(T_0 \cap w^{-1} I w) (U_{0} \cap w^{-1} I w)
\end{numequation}
be the induced Iwahori decomposition, cf. \cite[Lemma
    $3.3.2$]{OS1}. The Bruhat decomposition $G_0=\coprod_{w \in W} I w I$ induces a   decomposition 
  \begin{numeqnarray}\label{decomposition_Bruhat}
    D(G_0) \otimes_{D(\frg,U_0)} N  & \simeq &  \bigoplus_{w \in W}
     D(I) \otimes_{D(\frg, I \cap wU_0w^{-1})}  N^w \\\nonumber
     & \simeq &  \bigoplus_{w \in W}  D(w^{-1}I w) \otimes_{D(\frg, w^{-1}I w \cap U_0)} N 
      \end{numeqnarray}
      since $IwB_0=IwU_0.$
      For each $w \in W$, we have
  $$H^0(\fru_P, D( I ) \otimes_{D(\frg, I \cap wU_0w^{-1})} N^w)
    \simeq H^0(\mathrm{Ad}(w^{-1}) (\fru_P), D(w^{-1}I w) \otimes_{D(\frg,
      w^{-1}Iw \cap U_0)}  N).$$
  We can write each summand in the shape 
  $$\cN^w:=D(w^{-1}I w) \otimes_{D(\frg, w^{-1}I w \cap U_0)} N =
  \varprojlim_r \cN_r^w$$    where $\cN_r^w=D_r(w^{-1}I w) \otimes_{D(\frg, w^{-1}I w \cap U_0)} N$. If we denote
  by $N^w_r$ the topological closure\footnote{here we use the notation of \cite{OSch}.} of $N$ in $\cN^w_r$, we get by \cite[5.6.5]{OS2} finitely many elements  $u \in U^-_{0}$ and 
  finitely many elements  $v \in T_{0}$ such that
    \begin{numequation}\label{equation_split} \cN_r^w \simeq \bigoplus_{u,v} \delta_u \delta_v \otimes N_r^w
  \end{numequation}
  and the action of $\frx \in \mathrm{Ad}(w^{-1}) (\fru_P)$ is given by
  \begin{equation*}
    \frx \cdot \sum_{u,v} \delta_u\delta_v \otimes m_{u,v}=\sum \delta_u \delta_v \otimes
    \mathrm{Ad}((uv)^{-1}(\frx))m_{u,v}.
  \end{equation*}
Indeed, by \cite[Prop. 3.3.4]{OS1} we have isomorphisms 
  $$D_r(w^{-1}I w)\cong D_r(w^{-1}I w \cap U^-_{0})\hat{\otimes} D_r(w^{-1}I w \cap T_0)\hat{\otimes} D_r(w^{-1}I w \cap U_{0})$$ which induce isomorphisms
  $$D_r(\frg,w^{-1}I w\cap U_0)\cong U_r(\fru^-_{0} )\hat{\otimes} U_r(\frt) \hat{\otimes} D_r(w^{-1}I w \cap U_0)$$ where $U_r(\fru^-_{0})$ is the closure of 
  $U(\fru^-_{0} )$ in $D_r(w^{-1}I w \cap U^-_{0})$ etc. Hence we get 
  \begin{numeqnarray}\label{decomposition_Iwahori}
    \nonumber& & D_r(w^{-1}I w) \otimes_{D(\frg, w^{-1}I w \cap U_0)} N \\ & = & D_r(w^{-1}I w \cap U^-_{0}) \hat{\otimes} D_r(w^{-1}I w \cap T_{0}) \otimes_{U_r(\fru^-_{0})\otimes U_r(\frt)} N_r^w.
 \end{numeqnarray}
    But by the discussion in \cite[(5.5.7)- (5.5.8)]{OS2}
  there are for the fixed open subgroup $H$ (which is contained in $w^{-1}I  w$) finitely many elements $u \in U^-_{0}$ such that $D_r(H \cap U^-_{0})= \bigoplus_{u} \delta_u U_r(\fru^-_{0}).$ On the other hand there are by the compactness of $I$ finitely many $u \in U^-_{0}$  
  such that $D_r(w^{-1}I w \cap U^-_{0})=\bigoplus_{u} \delta_u D_r(H \cap U^-_{0})$ which gives the above claim.
  
  Fix $u\in U^-_0,v\in T_0,w\in W$ and
  abbreviate $\frh:=\mathrm{Ad}((uv)^{-1}(\mathrm{Ad}(w^{-1})(\fru_P)).$
  
  {\it Case 1:} $\frh \cap \fru^- \neq (0).$ Any non-trivial  element in this intersection acts injectively on $N$  and more generally on $\hat{N}$ by the presentation (\ref{presentation_N}). 
  Since $H^0(\frh,N^w_r)\subset H^0(\frh,\widehat{N})$ we conclude that $H^0(\frh, N_r^w)=0$. 
  
  {\it Case 2:} $\frh \subset \fru.$ 
 Then the parabolic subgroup $\mathrm{Ad}((uv)^{-1}(\mathrm{Ad}(w^{-1})(P)$ contains $B$ and thus there exists an element $x\in \frh$  attached to a simple root such that $\eta(x)\neq 0.$ Hence $x$ acts injectively on $N$ and more generally on $\widehat{N}$. In fact one checks this by elementary computations in $\widehat{N}$ or by applying \cite[Th, 4.1]{Kos2}.
    By the argument above in Case 1 we deduce again that $H^0(\frh, N_r^w)=0$.

  {\it Case 3:} In general  there is some element in $\frh$  which we can write as a sum $y+x$ with $y\in \fru^-$ and $x\in \fru.$
  Now let $v\in N_r^w$ and identify it with an element in  (\ref{presentation_N}). Then $y$ increases the degree whereas $x$ decreases the degree of the corresponding polynomial expression in the $y_\alpha$, $\alpha \in \Phi^+$. Hence $y+x$ acts injectively on $N_r^w$ and we argue as above.

  Hence by passing to the limit we
  get $H^0(\frh , \cN^w)=0$ for all $w\in W$. The claim follows.  
  
  ii) Here we have $P=B'$ and we consider the same case by case distinction as above.

  {\it Case 1:}  $\frh \cap \fru^- \neq (0).$ Any non-trivial  element in this intersection increases the degree of elements in $\widehat{N}$  by the presentation (\ref{presentation_N}). hence there cannot be a generalized eigenvector.
  
  {\it Case 2:} $\frh = \fru.$ By applying \cite[Th, 4.1]{Kos2}  that  $H^0(\fru_{\eta'},N)=0$ if $\eta'\neq \eta$ and $H^0(\fru_{\eta},N)=Kv^+$. The same holds true for the completions. In the latter case we get by $\delta_{wuv}\cdot v^+$ a generalized eigenvector: in loc.cit it is shown that
  $\fru_\eta$ is isomorphic as a $\fru$-module to the affine coordinate ring $K[U]$ of the affine variety $U$ with the natural $\fru$-module structure which explains the additional statements, as well.

  {\it Case 3:} In general there is some element in $\frh$  which we can write as a sum $y+x$ with $y\in \fru^-$ and $x\in \fru.$
  Now let $v\in N$ and identify it with an element in  (\ref{presentation_N}). Then $y$ increases the degree whereas $x$ decreases the degree of the corresponding polynomial expression in the $y_\alpha$, $\alpha \in \Phi^+$. Hence there is no eigenvector for the twisted action.

  Hence by passing to the limit we
  get the claim.

  %If there is some $g\in G_0$ with $Ad(g)(\fru',\eta')=(\fru,\eta)$ then $\delta_g\ast v^+ \subset H^0(\fru'_{\eta'},X_{\hat{\eta}}(N,V))$.
   \end{proof}

  \begin{cor}
   With the above notation we have $\overline{H}_0(U_{P,0},X_{\hat{\eta}}(N,V))=0.$ (where $\overline{H}_0$ denotes the Hausdorff completion of $H_0$).
  \end{cor}

  \begin{proof}
   This follows by duality, cf. also \cite[Theorem 3.5]{OSch}.
  \end{proof}

  \begin{prop}
 We have $\overline{H}_0(U_P,\cF^G_{\hat{\eta}}(N,V))=0.$
\end{prop}

  \begin{proof}
  %We have   $\overline{H}_0(U_P,c-\Ind^G_{G_0}(X_{\hat{\eta}}(N,V)')=\overline{H}_0(\fru_P, c-\Ind^G_{G_0}(X_{\hat{\eta}}(N,V)'))_{U_P}$. 
  Since $c-\Ind^G_{G_0}(X_{\hat{\eta}}(N,V)')$ is a direct sum $\bigoplus_{g\in G/G_0} gX_{\hat{\eta}}(N,V)'$ of copies of $X_{\hat{\eta}}(N,V)'$ and $\overline{H}_0(\fru_p,gX_{\hat{\eta}}(N,V)')=\overline{H}_0(Ad(g)(\fru_P),X_{\hat{\eta}}(N,V)')$  it suffices to see that $\overline{H}_0(\fru_P, X_{\hat{\eta}}(N,V)')=0$ since we consider from the very beginning arbitrary parabolic subgroups $P.$ Now we apply the above corollary.
\end{proof}

In the following we discuss an example which shows that the homological vanishing criterion of Kohlhaase is not satisfied in general. For this let $G=\GL_2.$

  \begin{prop}
 We have  $H_1(U^-,\cF^G_{\hat{\eta}}(N))\neq 0$.
\end{prop}

\begin{proof} By  \cite[Theorem 7.1]{K2} we have the identity $H_1(U^-,c-\Ind^G_{G_0}(X_{\hat{\eta}}(N)')=H_1(\fru^-, c-\Ind^G_{G_0}(X_{\hat{\eta}}(N)'))_{U^-}$. 
 Since $\dim \fru=1$ we conclude that $H_i(\fru^-, c-\Ind^G_{G_0}(X_{\hat{\eta}}(N)')$ is the cohomology of the complex
 $$y:C^1 \to C^0 $$ where $C^0=C^1=c-\Ind^G_{G_0}(X_{\hat{\eta}}(N)')$ and $y$ denotes multiplication by $y\in \fru^-.$ 
 
 Thus $H_1(\fru^-, c-\Ind^G_{G_0}(X_{\hat{\eta}}(N)'))$ coincides with
 $H^0(\fru^-, c-\Ind^G_{G_0}(X_{\hat{\eta}}(N)'))$ which is isomorphic
 to the topological dual of $\overline{H}_0(\fru^-, \prod_{G/G_0} X_{\hat{\eta}}(N).$ But 
  \begin{eqnarray*}                                                                                                    
H_0(\fru^-,\prod_{G/G_0}  X_{\hat{\eta}}(N)) & = & \prod_{G/G_0}  X_{\hat{\eta}}(N)/\fru^- \prod_{G/G_0} X_{\hat{\eta}}(N)\\ & = & \prod_{G/G_0}  X_{\hat{\eta}}(N)/\fru^- X_{\hat{\eta}}(N).
\end{eqnarray*}
Further
 $\fru^-D(G_0) = D(G_0)(\fru^-)^k N$ for some integer $k\geq 1$ since $\fru^-$ is a Ore set (cf. the proof of \cite[Lemma 4.2]{Ma}). But $N/(\fru^-)^k N$ is a finite-dimensional vector space.  In particular we see that $$D(G_0) \otimes_{U(\fru^-)} N/\fru^-D(G_0)\otimes_{U(\fru^-)} N=D(G_0)\otimes_{U(\fru^-)} (N/(\fru^-)^k N) $$ is a Hausdorff space. But the natural epimorphism
 $D(G_0) \otimes_{U(\fru^-)} N \to D(G_0) \otimes_{D(\frg,U_0)} N$ is strict. It follows by the obvious diagram chasing diagram that  $D(G_0) \otimes_{D(\frg,U_0)} N / \fru^-D(G_0) \otimes_{D(\frg,U_0)} N$ has to be Hausdorff, as well.
  Thus we conclude that
 $H_0(\fru^-, \prod_{G/G_0} X_{\hat{\eta}}(N)=\overline{H}_0(\fru^-, \prod_{G/G_0}  X_{\hat{\eta}}(N).$ Moreover, the image of the Whittaker vector in $N/(\fru^-)^k N$ does not vanish. It follows that  $D(G_0) \otimes_{D(\frg,U_0)} N / \fru^-D(G_0) \otimes_{D(\frg,U_0)} N\neq 0.$ Indeed by identify (\ref{decomposition_Bruhat}) we may suppose that $G_0$ is $I.$ Then we use the decomposition (\ref{decomposition_Iwahori}) in order to deduce the claim. This finishes the proof.
  \end{proof}

\begin{rmk}
 The same statements hold with the same proofs for the functors $\cF^G_{\hat{\eta},\chi}$.
\end{rmk}

\vspace{1cm}
\section{Irreducibility}

%In general we cannot expect the represenation $\cF_\mu(Z,V)$ to be topologically irreducible even if $Z$ and $V$ are irreducible due to the action of the centre $Z(G)$ of $G$. For this reason we consider the following variant of $\cF_\mu.$
%Suppose that $V$ are irreducible and let $\psi$ be a locally analytic character of $Z(G).$ 
%Then we set
%$$\overline{\cF}_\eta(\psi,V):={\rm c-\Ind}^G_{Z(G)G_0}((D(Z(G)G_0) \otimes_{D(\frg,U_0)}  (N_{\Theta,\eta}\otimes_K V')\otimes J_\psi')').?????????????????????????????$$

The first part of this chapter concerns the following irreducibility result. Here we fix as in the previous section a maximal ideal  $\Theta$ of $Z(\frg)$ and let $\eta$ be a non-degenerate  character with locally analytic lift $\hat{\eta}$ to $U_0$ and fix a locally analytic character of $Z(G).$

\begin{thm}\label{irred}  The locally analytic $G_0$-representation
$X_{\hat{\eta},\chi}(N_{\Theta,\eta})'$ is topologically irreducible. 
\end{thm}

\begin{proof}
We abbreviate $N_{\Theta,\eta}$ by $N$ and denote by $N_\chi:=N\otimes K_\chi'$ the corresponding $D(\frg,Z(G_0),U_0)$-module.  The proof is at the beginning verbatim the same as in  \cite{OS2,O2}  by replacing $P_0$ by $U_0$ and with the obvious (notational) changes.
We first prove that $X=X_{\hat{\eta},\chi}(N)$ is a simple $D(G_0)$-module and consider for an arbitrary open normal subgroup $H$  of $G_0$ the decomposition $$X=\bigoplus_{g\in G_0/H} \delta_g \star \left(D(HZ(G_0)U_0) \otimes_{D(\frg,Z(G_0)U_0)} N_\chi \right).$$ Here we recall  that for a $D(H)$-module $N$ and $g \in G_0,$ we denote by $\delta_g \star N$ the space $N$, equipped with the structure of a $D(H)$-module given by $\delta \cdot_g n = (\delta_{g^{-1}} \delta \delta_g)n$, where $n \in N$, $\delta \in D(H)$, and the product $\delta_{g^{-1}} \delta \delta_g \in D(G_0)$ is contained in $D(H)$.

We arrive at the following modified version of Thm. 5.5  loc.cit. is a reduced version which proves that $X_{\hat{\eta},\chi}(N)$ is a  simple $D(G_0)$-module.

\begin{thm}\label{irredH}  Let $H$ be an open normal subgroup of $G_0$, and let $g, g_1, g_2 \in G_0$. Then

(i) The $D(H)$-module $\delta_g \star \left(D(HZ(G_0)U_0) \otimes_{D(\frg,Z(G_0)U_0)} N_\chi\right)$ is simple.

(ii) The $D(H)$-modules $\delta_{g_i} \star \left(D(HZ(G_0)U_0) \otimes_{D(\frg,Z(G_0)U_0)} N_\chi\right),$ $i=1,2,$ are isomorphic if and only if $g_1 HZ(G_0)U_0 = g_2HZ(G_0)U_0$.
\end{thm}

\noindent \Pf  The proof passes through several reduction steps. The first step is to reduce to the case of a suitable subgroup $H_0 \sub H$ which is normal in $G_0$ and uniform pro-$p$.

{\it Step 1: reduction to $H_0$.}  Let $\Lie(\bG_0)$, $\Lie(\bU_0)$ etc.  be the Lie algebras of $G_0$, $U_0$ etc. These are $O_L$-lattices in $\frg = \Lie(G)$, $\fru = \Lie(U)$ etc. respectively. Moreover, $\Lie(\bG_0)$, $\Lie(\bU_0)$ etc.  are $\Zp$-Lie algebras, and we have as usual a decomposition $\Lie(\bG_0)= \Lie(\bU_{0}) \oplus \Lie(\bT_0) \oplus  \Lie(\bU^-_{0}). $

\noindent For $m_0 \ge 1$ ($m_0 \ge 2$ if $p=2$) the $O_L$-lattices $p^{m_0} \Lie(\bG_0)$, $p^{m_0} \Lie(\bU_0)$ etc.  are powerful $\Zp$-Lie algebras, cf. \cite[sec. 9.4]{DDMS}, and hence
$\exp_G: \frg \dashrightarrow G$ converges on these $O_L$-lattices. We set 
\begin{numequation}\label{H_0}
H_0 = \exp_G\Big(p^{m_0} \Lie(\bG_0)\Big) \;, \;\; H_0^0 = \exp_G\Big(p^{m_0} \Lie(\bT_0)\Big), 
\end{numequation}
\begin{equation*} H_0^+ = \exp_G\Big(p^{m_0} \Lie(\bU_{0})\Big), H_0^- = \exp_G\Big(p^{m_0} \Lie(\bU_{0}^-)\Big),  
\end{equation*}
which are uniform pro-$p$ groups. Moreover $H_0$ is normal in $G_0$ and  $H_0^+$ in $U_0$, respectively. 
We choose   $m_0$ large enough such that $H_0$ is contained $H$.
With the same argument as in \cite{OS2} we may assume from now on that $H_0=H.$

{\it Step 2: passage to $D_r(H)$.} We put
$$\kappa = \left\{\begin{array}{lcl} 1 & , & p>2 \\
2 & , & p=2 \end{array} \right.$$
and denote by $r$ a real number in $(0,1) \cap p^\bbQ$ with the property that there is $m \in \bbZ_{\ge 0}$ such that $s = r^{p^m}$ satisfies
%\footnote{This asssumption seems to be also necessary in \cite{OS2}}
\begin{numequation}\label{r_and_s}
 \max\{\frac{1}{p},|\pi|p^{-1/(p-1)} \} < s \mbox{ and } s^{\kappa} < p^{-1/(p-1)} 
\end{numequation}
where $\pi\in O_L$ is an uniformizer.
We let $\|\cdot\|_r$ denote the norm on $D(H)$ associated to the canonical $p$-valuation. Here $D_r(H)$ is the corresponding Banach space completion. This is a noetherian Banach algebra, and $D(H) = \varprojlim_{r <1} D_r(H)$. Set $\bN:=D(HZ(G_0)U_0)\otimes_{D(\frg,Z(G_0)U_0)}N_\chi.$ As explained in loc.cit. it is enough to show that
$$\bN_r:= D_r(H) \otimes_{D(H)} \bN = D_r(HZ(G_0)U_0) \otimes_{D(\frg,Z(G_0)U_0)} N_\chi$$
are simple $D_r(H)$-modules for a sequence of $r$'s converging to $1$. From now on we assume that, in addition to (\ref{r_and_s}), that $r$ is chosen such that $\bN_r \neq 0$, and we consider $N$ as being contained in $\bN_r$.

{\it Step 3: passage to $U_r(\frg)$.} Let $U_r(\frg) = U_r(\frg,H)$ be the topological closure of $U(\frg)$ in $D_r(H)$. Then
\begin{numequation}\label{density1}
\hskip-50pt \mbox{$U(\frg)$ is dense in $D_r(H)$ if $r^{\kappa} < p^{-\frac{1}{p-1}}$ and $U_r(\frg) = D_r(H)$.}
\end{numequation}
 Let $(P_m(H))_{m \ge 1}$ be the lower $p$-series of $H$, cf. \cite[1.15]{DDMS}. Note that $P_1(H) = H$. For $m \ge 0$ put $H^m : = P_{m+1}(H)$ so that $H^0 = H$. Then $H^m$ is a uniform pro-$p$ group with $\Zp$-Lie algebra equal to $p^m\Lie_\Zp(H)$. 
 Thus our subgroups above have the following description.
 Let $\frg_\bbZ$ be a $\bbZ$-form of $\frg$, i.e. $\frg_\bbZ \otimes_\bbZ L = \frg$. We fix a Chevalley basis $(x_\gamma,y_\gamma,h_\alpha \midc \gamma \in \Phi^+, \alpha \in \Delta)$ of $[\frg_\bbZ,\frg_\bbZ]$. We have $x_\gamma \in \frg_\gamma$, $y_\gamma \in \frg_{-\gamma}$, and $h_\alpha = [x_\alpha,y_\alpha] \in \frt$, for $\alpha \in \Delta$. Then
$$\Lie_\Zp(H^-) = p^{m_0}\Lie(\bU^-_{0}) = \bigoplus_{\beta \in \Phi^+} O_L y_\beta^{(0)} \;,$$
where $y_\beta^{(0)} = p^{m_0}y_\beta$. Moreover, the $\Zp$-Lie algebra of $H^{-,m}:=H^m\cap H^-$ is
$$\Lie_\Zp(H^{-,m}) = p^m\Lie(H^-) = \bigoplus_{\beta \in \Phi^+} O_L y_\beta^{(m)} \,,$$
where $y_\beta^{(m)} = p^{m_0+m}y_\beta$.
In the same way
$$\Lie_\Zp(H^{+,m}) = p^m\Lie(H^+) = \bigoplus_{\beta \in \Phi^+} O_L x_\beta^{(m)} \,,$$
where $x_\beta^{(m)} = p^{m_0+m}x_\beta$
and
$$\Lie_\Zp(H^{0,m}) = p^m\Lie(H^0) = \bigoplus_{\beta \in \Phi^+} O_L h_\beta^{(m)} \,,$$
where $h_\beta^{(m)} = p^{m_0+m}h_\beta$.

 Let $s = r^{p^m}$ be as in (\ref{r_and_s}). Denote by $\|\cdot\|_s^{(m)}$ the norm on $D(H^m)$ induced by the canonical $p$-valuation on $H^m$. Then, by \cite[6.2, 6.4]{Sch}, the restriction of $\|\cdot\|_r$ on $D(H)$ to $D(H^m)$ is equivalent to $\|\cdot\|_s^{(m)}$, and $D_r(H)$ is a finite and free $D_s(H^m)$-module on a basis any set of coset representatives for $H/H^m$. By (\ref{density1}) we can conclude:
\begin{numequation}\label{density2}
\mbox{If $s = r^{p^m}$ is as in (\ref{r_and_s}), then $U(\frg)$ is dense in $D_s(H^m)$, hence $U_r(\frg,H) = D_s(H^m)$.}
\end{numequation}
In particular, $U_r(\frg) \cap H = H^m$ is an open subgroup of $H$. Let
$$\frn_r := U_r(\frg)N$$
\noindent be the $U_r(\frg)$-submodule of $\bN_r$ generated by $N$. The module $N$ is dense in $\frn_r$ with respect to this topology. Now $N$ is essentially a  finite direct sum of copies of $U(\fru^-)$ by the Harish-Chandra isomorphism (see also identity (\ref{presentation_N}) below). Hence it follows from \cite[1.3.12]{F}
or \cite[3.4.8]{OS1} that $\frn_r$ is a simple $U_r(\frg)$-module, and in particular a simple $D_r(\frg,Z(G_0)U_0)$-module. Thus for every $g \in G_0$ the $U_r(\frg)$-module $\delta_g \star \frn_r$ (which is defined as above by composing the natural action with conjugation with $g$) is simple. As in \cite{OS2} we have for $U_{0,r}:= HZ(G_0)U_0 \cap D_r(\frg,Z(G_0)U_0)$ 

\medskip
\noindent (i) $U_{0,r} = H^m Z(G_0) U_0=Z(G_0)U_0H^{0,m} H^{-,m}$.

\noindent (ii) $D_r(HZ(G_0)U_0) =  \bigoplus_{g \in HZ(G_0)U_0/U_{0,r}} \delta_g D_r(\frg,Z(G_0)U_0)$.

\medskip
By (ii) we deduce that  
$$\frn_r = D_r(\frg,U_0) \otimes_{D(\frg,Z(G_0)U_0)} N \sub D_r(HZ(G_0)U_0) \otimes_{D(\frg,Z(G_0)U_0)} N = \bN_r$$ and
\begin{numeqnarray}\label{decomp}
\bN_r & = & D_r(HZ(G_0)U_0) \otimes_{D(\frg,Z(G_0)U_0)} N = D_r(HZ(G_0)U_0) \otimes_{D_r(\frg,Z(G_0)U_0)} \frn_r \\ \nonumber  & = & \bigoplus_{g \in HZ(G_0)U_0/U_{0,r}} \delta_g \frn_r \,.
\end{numeqnarray}
Here the action of $U_r(\frg)$ on  $\delta_g \frn_r$ is the same as on $\delta_g \star \frn_r.$
Thus it suffices to show the theorem below.
Thus for proving Theorem \ref{irredH} it suffices to apply the next theorem.
 
\begin{thm}\label{U_r_modules} Assume that all the  above assumptions are satisfied. Then
for any $g, h \in G_0$ with $gU_{0,r} \neq hU_{0,r}$ the $U_r(\frg)$-modules $\delta_{g} \star \frn_r$ and $\delta_{h} \star \frn_r$ are not isomorphic.
\end{thm}

\noindent \Pf We start with the following observation.
By our assumption on $r$ and $s$ we have $U_r(\fru^-) = U_r(\fru^-,H^-) = D_s(H^{-,m})$ by (\ref{density2}). Elements in $U_r(\fru^-)$ thus have a description as power series in $(y^{(m)}_\beta)_{\beta \in \Phi^+}$:
$$U_r(\fru^-) = \left\{\sum_{n = (n_\beta)} d_n  (y^{(m)})^n \midc \lim_{|n| \ra \infty} |d_n|s^{\kappa|n|} = 0 \right\} \,,$$
where $(y^{(m)})^n$ is the product of the $(y^{(m)}_\beta)^{n_\beta}$ over all $\beta \in \Phi^+$, taken in some fixed order.
Let $\|\cdot\|_s^{(m)}$ be the norm on $D_s(H^{-,m})$ induced by the canonical $p$-valuation on $H^{-,m}$. Then we have for any generator $y^{(m)}_\beta$
\begin{numequation}\label{norm_generator}
\| y^{(m)}_\beta\|_s^{(m)} = s^\kappa \;.
\end{numequation}
%By symmetry the discussion above holds true for the group $H^{+,m}.$ {\bf eher fuer $T_o$ hinschreiben!!!!}Thus we may write
%\begin{align*}\frm_r & = & \\ & & \left\{\sum_{n = (n_\beta)} d_n  (y^{(m)})^n v_\mu + \sum_{n = (n_\beta)} c_n  (x^{(m)})^nv_\mu\midc \lim_{|n| \ra \infty} \|d_n (y^{(m)})^n\|_s^{(m)} = 0, \lim_{|n| \ra \infty} \|c_n (x^{(m)})^n\|_s^{(m)} = 0 \right\} \, 
%\end{align*}
%where $v_\mu\in M$ is a fixed generator of weight $\mu$.

Now let $\phi: \delta_{g} \star \frn_r \stackrel{\simeq}{\lra} \delta_{h} \star \frn_r$ be an isomorphism of $U_r(\frg)$-modules. We may assume $h = 1$. Let $I \sub G_0$ be the Iwahori subgroup. Using the Bruhat decomposition
$$G_0 = \coprod_{w \in W} IwB_0$$ 
\noindent we may write $g = k_1wk_2$ with $k_1\in I$,  $w \in W$ and $k_2 \in B_0$. Since $T_0\subset I \cap B_0$ and $w$ normalizes $T_0$ we may suppose that $k_2 \in U_{0}.$ As 
$\delta_u\star \frn_r=\frn_r$ for $u\in U_{0}$ we may then suppose that $k_2=e.$
By the Iwahori decomposition 
$$I = (I \cap {B_{0}^-}) \cdot (I \cap U_0)$$
we may write $k_1=hu^+$ with $u^+ \in  I \cap {U_{0}}$ and $h \in I \cap {B_0^-} $.
By shifting the factors from the left to the right it suffices to consider an isomorphism
$$\phi: \delta_w \star \frn_r \stackrel{\simeq}{\lra} \delta_h \star \frn_r \,.$$
with $h\in B^-_0.$

Suppose that $w\neq 1$. Then there is a positive root  $\beta \in \Phi^+$ such that $w^{-1}(\beta) < 0$, hence $w^{-1}(-\beta) > 0$, cf. \cite[2.3]{Car}. Consider a non-zero element element $y \in \frg_{-\beta}$, and let $v^+ \in N$ be a generator (the unique Whittaker vector) of $N.$ Then we have the following identity in $\delta_w \star \frn_r$,
$$y \cdot_w v^+ = \Ad(w^{-1})(y) \cdot v^+ = \eta(\Ad(w^{-1})(y)) v^+.$$
On the other hand,  we have $\phi(v^+) = v$ for some nonzero $v \in \frn_r$. And therefore 
$$\eta(\Ad(w^{-1})(y)) v= \phi(y \cdot_w v^+) = y \cdot_h \phi(v^+) = y \cdot_h v = \Ad(h^{-1})(y) \cdot v \;.$$
%Consider $M_\eta=U(\frg)\otimes_{U(\fru)} K_\eta$. By Poincar\'e Birkhoff Witt this is a isomorphic to $U(\frb^-)$ as a $U(\frb^-)$-module in particular it is free.
Since $h\in B^-_0$ we see that $\Ad(h^{-1})(y)\in \fru^-.$  
But $U(\fru^-)$ acts freely on $N$ since $U(\frg)$ is free over $Z(\frg)\otimes U(\fru)$, cf. Theorem \ref{free}.
It follows that $U(\fru^-)$ acts freely on $\frn_r$, as well. Thus  $\Ad(h^{-1})(y) \cdot v$
cannot be a scalar multiple of $v.$

So we may suppose that $w=1$ and that there is an isomorphism
$$\phi: \frn_r \stackrel{\simeq}{\lra} \delta_h \star \frn_r \,$$
with $h\in B_0^-.$ Write $h=tu$ with $t \in T_0$ and $u \in U^-_{0}.$ By shifting $t$ to the left hand side we get an isomorphism
$$\phi: \delta_t \star \frn_r \stackrel{\simeq}{\lra} \delta_u \star \frn_r \,.$$
Let again $v=\phi(v^+).$ 
For $x \in \fru$, we get 
\begin{numequation}\label{equation_Whittaker}
\eta(\Ad(t^{-1})(x)) v= \phi(x \cdot_t v^+) = x \cdot_u \phi(v^+) = x \cdot_u v = \Ad(u^{-1})(x) \cdot v \;.
 \end{numequation}
Thus $v$ is an eigenvector for the action of 
 $\Ad(u^{-1})(\fru)$. By Proposition \ref{H^0_2} ii) this implies that $t$ is in the centre and up to scaling that $v=\delta_u \cdot v^+.$\footnote{Here is a minor difference in the argumentation compared to \cite{OS2}} Hence we may assume that $t=1$ and we may write
$v = \delta_{u^{-1}} \cdot v^+$ with
\begin{numequation}\label{series}
\delta_{u^{-1}} \cdot v^+ = \sum_{n \ge 0} \frac{1}{n!} (-\log_{U_P^-}(u))^n \cdot v^+ =: \Sigma \hskip5pt .
\end{numequation}

As in loc.cit. we shall show that the series $\Sigma$, which is an element of $\hat{N}$, does in fact not lie in the image of $\frn_r$ in $\hat{N}$, if $u \notin U_{0,r}$.
For this we write
$$\log_{U_P^-}(u) = \sum_{\beta \in B(u)} z_\beta \;,$$
\noindent with a non-empty set $B(u) \sub \Phi^+$ and non-zero elements $z_\beta \in \frg_{-\beta}$. Put
$$B^+(u) = \{\beta \in B(u) \midc z_\beta \notin O_L y^{(m)}_\beta \} \;.$$
\noindent This is a non-empty subset of $B(u)$ since $u \notin U_{0,r}$. Put $B'(u) = B(u) \setminus B^+(u)$,
$$z^+ = \sum_{\beta \in B^+(u)} z_\beta \hskip10pt \mbox{ and } \hskip10pt z' = \sum_{\beta \in B'(u)} z_\beta = \log_{U^-}(u) - z^+ \;.$$
\noindent Then $z' \in \Lie_\Zp(H^{-,m})$, and thus $\exp(z') \in H^{-,m}$. The element $u_1 = u\exp(z')^{-1} = u\exp(-z')$ does not lie in $H^{-,m}$, and $\delta_u \star \frn_r \cong \delta_{u_1} \star \frn_r$. We may hence replace $u$ by $u_1 = u \exp(-z')$. By iterating this process we can assume that $B'(u)=\emptyset$, cf. loc.cit.

Next we chose as in loc.cit. an extremal element $\beta^+$ among the $\beta \in B(u) = B^+(u)$, i.e. one of the minimal generators of the cone $\sum_{\beta \in B(u)} \bbR_{\ge 0}\beta$
inside $\bigoplus_{\alpha \in \Delta} \bbR \alpha$. Then 
if we write $\log_{U^-}(u))=z_{\beta^+} + z_{\beta_2} + \ldots + z_{\beta_k}$ where $B(u) = \{\beta^+, \beta_2, \ldots, \beta_k \}$ it suffices to see that the series $\frac{(-1)^n}{n!}z_{\beta^+}^n \cdot v^+$ does not converge to an element in $\frn_r$.

Because $B'(u)=\emptyset$, we have $z_{\beta^+} = \alpha\pi^{-k}y^{(m)}_{\beta^+}$ for some $k\geq 1$ and $\alpha \in O_L^\ast.$ We then get 
$$||(-1)^n\frac{z_{\beta^+}^n}{n!}||_s^{(m)} =  |\pi^{-k}| \frac{s^n}{|n!|}=|\frac{\pi^{-n(k-1)}}{n!}| |(\frac{s}{\pi})|^n.$$ By identity (\ref{r_and_s}), we have $|\frac{s}{\pi})|>p^{-\frac{1}{(p-1)}} $ which is the radius of convergence of the exponential series. As further $|\pi^{-1}|>1$ we see that $\Sigma$ does not converge. This follows.
\end{proof}

\begin{thm}\label{irred_2}  The locally analytic $G$-representation
$\cF^G_{\hat{\eta},\chi}(N_{\Theta,\eta})$ is topologically irreducible. 
\end{thm}

\begin{proof}
We abbreviate again $X=X_{\hat{\eta},\chi}(N_{\Theta,\eta})$.
By the irreducibility criterion of Bode \cite[Thm. 4.1]{B} we need to show that for $t\in Z(G)G_0\backslash G/G_0$ with $t\neq 1$  there are no non-trivial homomorphism $X'_{\mid G_0\cap tG_0t^{-1}} \to 
 t\cdot X'_{\mid G_0\cap tG_0t^{-1}}.$\footnote{The modules $X$ are uniformly simple by the proof Theorem \ref{irred}.}
 We choose again a normal open subgroup $H$ of $G_0$ contained in $G_0\cap tG_0t^{-1}$ and show that there is even no such homomorphism as $H$-representations. Here we pass to the dual side and consider the decomposition of $X$ as $D(H)$-module:
 $X_{\mid H}=\bigoplus_{g\in G_0/H} \delta_g  \star D(HZ(G_0)U_0) \otimes_{D(\frg,Z(G_0)U_0)} N$ of $D(H)$-modules. 
 It suffices then to show that for each  $g \in G_0/H$ there is no homomorphism 
 $\phi:\delta_g \star D(HZ(G_0)U_0) \otimes_{D(\frg,Z(G_0)U_0)} N \to \delta_t \star D(G_0)\otimes_{D(\frg,Z(G_0)U_0)} N$  of $U(\frg)$-modules.

%Wemay assume that $g=1.$ Indeed, if there is a homomorphism 
%$\delta_{g} D(H) \otimes_{D(\frg,U_0\cap H)} N \to \delta_s  D(F) \otimes_{D(\frg,U_0\cap H)} N$
% of $D(H)$-modules, then we multiply both sides  by $\delta_{g^{-1}}$ and obtain a 
% $D(H) \otimes_{D(\frg,U_0\cap H)} N \to \delta_{g^{-1}}\delta_s  D(F) \otimes_{D(\frg,U_0\cap H)} N$
% of $D(H)$-modules. But $\delta_{g^{-1}}\delta_s=\delta_{s'} \delta_{g^{-1}}$ and  
% $\delta_{g^{-1}} X= X.$ 

%Consider   the formal completion  $\hat{N}$ of $N.$ Then $$H^0(\frn,N)= H^0(\frn,\hat{N})=Kv^+.$$ Since $\hat{N}$ is also the formal completion of $\delta_t\star N$ and $\delta_t \star \eta$ is also a non-degenerate Whittaker function for $v^+$  we assume that $\phi(v^+)=v^+.$ But the

Let $x\in Ad(g^{-1})(\fru).$ Then $x':=Ad(g)(x)\in \fru$ and thus 
$Ad(g)(x)v^+=\eta(x')v^+.$ Hence
$$\eta(x')\phi(v^+)=\phi(Ad(g)(x)v^+)= Ad(t)(x)\phi(v^+).$$ Thus $\phi(v^+)$ is an eigenvector for the action of $Ad(t\cdot g^{-1})(\fru).$
By Proposition \ref{H^0_2} ii) it follows that $t=1$ which gives a contradiction if $\phi(v^+)\neq 0.$
\begin{comment} 
Now let $W\subset c-\Ind^G_{G_0}(X_{\hat{\eta}}(N)')$ be a closed subrepresentation. Then since there are no non-trivial homomorphisms between the direct summands in (\ref{Mackey}) we have 
\begin{numequation}\label{decomposition}
W_{\mid G_0} =  \bigoplus_{t \in S} W_t
\end{numequation} for certain subrepresentations $W_t\subset  \Ind^{G_0}_{G_0 \cap  tG_0t^{-1}} \big(t{X_{\hat{\eta}}(N)'} \big).$

On the other hand, we consider the composition of the two innclusions $W\hookrightarrow  c-\Ind^G_{G_0}(X') \hookrightarrow \Ind^G_{G_0}(X').$ By Frobenius reciprocity we thus get a map $W_{\mid G_0} \to X'$. If this map vanishes then we have $W=0.$ If not, then we see by the Mackey decomposition and (\ref{decomposition}) that $W$ has to contain $X'$. But then it contains all translations$gX'$, as well, and therefore $W=c-\Ind^G_{G_0}(X_{\hat{\eta}}(N)').$
\end{comment}
\end{proof}

 \vskip8pt
 
% \begin{rmk}
%A natural generalization of the above theorem is to consider also a smooth contribution $V$ in our functor.
%By functoriality $V$ has to be irreducible in order that
%$\cF^G_{\hat{\eta}}(N_{\Theta,\eta},V)$ is topologically irreducible as $G$-representation. However, an irreducible smooth admissible representation of $U_0$ has to be 1-dimensional, hence it is given by a smooth character $\chi$ of $U_0.$ In this case one verifies the irreducibility as presented in the proof of \cite[Theorem 5.1]{O2}.
%\end{rmk}

A natural generalization of the above theorem is to consider also a smooth contribution $V$ in our functor.
By functoriality $V$ has to be irreducible in order that
$\cF^G_{\hat{\eta}}(N_{\Theta,\eta},V)$ is topologically irreducible as $G$-representation. This condition is also sufficient:
 
\begin{thm} Let $\hat{\eta}$ and $\Theta$ be chosen as above. Let $V$ be a smooth irreducible representation of $U_0$. Then $\cF^G_{\hat{\eta},\chi}(N_{\Theta,\eta},V)$ is topologically irreducible as $G$-representation.
\end{thm}

\begin{proof} 
We prove first as in Theorem  \ref{irred} that the representation is topologically irreducible. Here the proof follows the same lines as in \cite{OS2}
We present a compactified version and refer for more details to loc.cit.

Let $X= X_{\hat{\eta},\chi}(N_{\Theta,\eta},V)$ and $\frd$ be the system of differential equations as in Example \ref{example_Theta}. Again we show first that $X'$ is topologically irreducible.
For a closed $G_0$-subrepresentation $W \sub X'$, we put
$$W_{sm} = \varinjlim_H \Hom_H(\Ind_{Z(G_0)U_0}^{G_0}(K_{\hat{\eta},\chi}')^\frd|_H, W|_H) \,,$$
where the limit is taken over all compact open subgroups $H$ of $G_0$. Then $G_0$ acts on $W_{sm}$  via 
$(g\phi)(f) = g(\phi(g^{-1}(f))$
and yields by construction a smooth representation. 
We consider the continuous map of $G_0$-representations
$$\Phi_W: \Ind_{Z(G_0)U_0}^{G_0}(K_{\hat{\eta},\chi}')^\frd \otimes_K W_{sm} \lra W \,,
\; f \otimes \phi \mapsto \phi(f) \,.$$
In the following lines we want to describe $X'_{sm}$. 
Let $H \sub G_0$ be an open normal subgroup. We have
\begin{numequation}\label{generaldecomp}
\Ind_{Z(G_0)U_0}^{G_0}(K_{\hat{\eta},\chi}')^\frd|_H \; = \bigoplus_{\gamma \in H \backslash G_0/Z(G_0)U_0}
\Ind^H_{H \cap \gamma Z(G_0)U_0 \gamma^{-1}}((K_{\hat{\eta},\chi}')^\gamma)^\frd 
\end{numequation}
where  $(K_{\hat{\eta},\chi}')^\gamma$ is $K_{\hat{\eta},\chi}$ with the twisted action.

\noindent By Theorem \ref{irredH}  all $H$-representations on the right hand side of (\ref{generaldecomp}) are topologically irreducible and pairwise non-isomorphic. Similarly, for the representation $X'=\Ind_{Z(G_0)U_0}^{G_0}(K_{\hat{\eta},\chi}' \otimes_K V)^\frd$ we have
$$\Ind_{Z(G_0)U_0}^G(K_{\hat{\eta},\chi}' \otimes_K V)^\frd|_H = \bigoplus_{\gamma \in H \backslash G_0/Z(G_0)U_0}
\Ind^H_{H \cap \gamma Z(G_0)U_0 \gamma^{-1}}((K_{\hat{\eta},\chi}')^\gamma \otimes_K V^\gamma)^\frd \,.$$

\noindent Denote by $V^{H \cap Z(G_0)U_0} \sub V$ the subspace of fix vectors for  $H \cap Z(G_0)U_0$. Then
$$\left(\Ind^H_{H \cap \gamma Z(G_0)U_0 \gamma^{-1}}((K_{\hat{\eta},\chi}')^\gamma)\right)^\frd \otimes_K V^{H \cap Z(G_0)U_0} \cong\Ind^H_{H \cap \gamma Z(G_0)U_0 \gamma^{-1}}((K_{\hat{\eta},\chi}')^\gamma \otimes_K (V^{H \cap Z(G_0)U_0})^\gamma)^\frd$$ is a subspace of  $\Ind^H_{H \cap \gamma Z(G_0)U_0 \gamma^{-1}}((K_{\hat{\eta},\chi}')^\gamma \otimes_K V^\gamma)^\frd.$
Let
$$\phi: \Ind_{Z(G_0)U_0}^{G_0}(K_{\hat{\eta},\chi}')^\frd|_H \ra \Ind_{Z(G_0)U_0}^{G_0}(K_{\hat{\eta},\chi}' \otimes_K V )^\frd|_H$$
\noindent be a continuous homomorphism of $H$-representations. For each $\gamma,$ let $f_\gamma \in \Ind^H_{H \cap \gamma Z(G_0)U_0 \gamma^{-1}}((K')^\gamma)^\frd$
be a non-zero element.  Because of the irreducibility the map $\phi$ is uniquely determined by the images $\phi(f_\gamma)$ which are rigid-analytic functions on the components of some covering of $G_0$ by 'polydiscs' $\Delta_i$ and which take values in a $BH$-subspace of $K' \otimes_K V$. But any such $BH$-subspace is finite-dimensional, hence contained in a subspace of the form $K' \otimes_K V^{H' \cap Z(G_0)U_0}$ for some open  subgroup $H' \sub H$ which is normal in $G_0$. Shrinking $H'$ we may further suppose that each $\Delta_i$ is $H'$-stable and that each $\phi(f_\gamma)$ takes values in $K' \otimes_K V^{H' \cap Z(G_0)U_0}$ and 
$$\phi(h.f_\gamma)(g) = [h.\phi(f_\gamma)](g) = \phi(f_\gamma)(h^{-1}g) \in K' \otimes_K V^{H'\cap Z(G_0)U_0}$$
for any $h \in H$. Because the subspace of $\Ind^H_{H \cap \gamma Z(G_0)U_0 \gamma^{-1}}((K')^\gamma)^\frd$ generated by the functions $h.f_\gamma$, $h \in H$, is dense in $ \Ind^H_{H \cap \gamma Z(G_0)U_0 \gamma^{-1}}((K')^\gamma)^\frd$, we have $\phi(f)(g) \in K_\eta' \otimes_K V^{H' \cap Z(G_0)U_0}$ for all $f \in  \Ind^H_{H \cap \gamma Z(G_0)U_0 \gamma^{-1}}((K')^\gamma)^\frd$. It follows that $\phi$ induces a continuous map of $H'$-representations
$$\phi: \Ind_{Z(G_0)U_0}^{G_0}(K_{\hat{\eta},\chi}')^\frd|_{H'} \ra \left[\Ind^{G_0}_{Z(G_0)U_0} (K_{\hat{\eta},\chi}')^\frd \right]|_{H'} \otimes_K V^{H' \cap Z(G_0)U_0}$$
Here the irreducible $H'$-representation $\Ind^{H'}_{H' \cap \gamma Z(G_0)U_0 \gamma^{-1}}((K_{\hat{\eta},\chi}')^\gamma)^\frd$ contained in the left hand side is mapped via $\phi$ to  $\left[\Ind^{H'}_{H' \cap \gamma Z(G_0)U_0 \gamma^{-1}}((K_{\hat{\eta},\chi}')^\gamma)^\frd\right] \otimes_K V^{H' \cap Z(G_0)U_0}$ and the restriction of $\phi$ to this summand has the shape $f \mapsto f \otimes v$ for some $v \in V^{H' \cap Z(G_0)U_0}$. 
The vectors $v$ give rise to an element in the smoothly compact induction  $c-\ind_{Z(G_0)U_0}^{G_0}(V)$ where $Z(G_0)$ acts trivially on $V$. Further one deduces
as in loc.cit.   from this that there is a canonical isomorphism 
\begin{eqnarray*}
 c-\ind_{Z(G_0)U_0}^{G_0}(V) & \ra &  X_{\hat{\eta},\chi}(N,V)'_{sm}.
\end{eqnarray*}
and  that  $W_{sm} \neq \{0\}$ for any non-zero closed $G$-invariant subspace $W \sub X_{\hat{\eta},\chi}(N,V)'$.   
From the following claim we see that $X_{\hat{\eta},\chi}(N,V)'$ is topologically irreducible.

\noindent  Claim: $\Ind^{G_0}_{U_0}(K_\eta')^\frd \otimes_K W_{sm}$ surjects onto $X_{\hat{\eta}}(N,V)'$. 
 
 %For this it suffices to show that
 %$\Ind^{G_0}_{U_0}(K_\eta')^\frd \otimes_K W^0_{sm}$ surjects onto $X(N_{\Theta,\eta},V)'$ where $W^0_{sm}:=W_{sm}\cap \ind^{G_0}_{U_0}(V)$.
 
For this consider the $U_0$-morphism 
$\Pi: W_{sm} \ra V$ given by $\varphi \mapsto \varphi(1)$. As $W_{sm}\neq (0)$ this map is surjective.
Therefore, for any $v \in V$ there is some $\varphi \in W_{sm}$ with $\varphi(1) = v$. As $W_{sm}$ is a $G_0$-representation we may replace $1$ by any  $g \in G_0$, i.e.  $\varphi(g) = v$. Then, on a neighborhood $N$ of $g$, the map $\varphi$ is constant with value $v$.

Let $S \sub G_0$ be a compact locally analytic subset such that $S \cdot Z(G_0)U_0=G_0$. Let $S' \sub S$ and $C\subset Z(G_0)U_0$  be compact open subsets such that $S' \cdot C \subset N$. Let $f: G_0 \ra K'$ be a function whose support is contained in $S' \cdot C$. Then the function $S \ra K' \otimes_K V$, $s \mapsto f(s) \otimes \phi(s)$, is equal to the function $S \ra K' \otimes_K V$, $s \mapsto f(s) \otimes v$.

\vskip5pt

Let $f \in X'=\Ind^{G_0}_{Z(G_0)U_0}(K_{\hat{\eta},\chi}' \otimes_K V)^\frd$ be any element. Then $f$ is uniquely determined by its restriction to $S$.
As we have seen above, the set $f(G_0)$ is contained in a finite-dimensional vector space 
$K_{\hat{\eta},\chi}' \otimes_K V_0$ with $V_0 \sub V$. Let $v_1, \ldots, v_r$ be a basis of $V_0$. Write 
$f(s) = f_1(s) \otimes v_1 +\ldots + f_r(s) \otimes v_r$ with locally analytic $K_{\hat{\eta},\chi}'$-valued functions $f_i$ on $S$. 
Extend each function $f_i$ to $G_0$ by $f_i(s\cdot z \cdot u) = \hat{\eta}(u^{-1})\chi(z)f_i(s)$. 
Then $f_i \in \Ind^{G_0}_{Z(G_0)U_0}(K_{\hat{\eta},\chi}')$ for all $i$. In fact, one verifies  that 
$f_i \in \Ind^{G_0}_{Z(G_0)U_0}(K_{\hat{\eta},\chi}')^\frd$ for all $i$. For any $s \in S$, choose some $\varphi_{i,s} \in W_{sm}$ such that $\varphi_{i,s}(x) = v_i$ for all $x$ in a compact open neighborhood $N_{i,s} \sub S$ of $s$. As $S$ is compact, finitely many of the $N_{i,s}$ will cover $S$. 
We can then choose a (finite) disjoint refinement $(N_{i,j})_j$ of the finite covering. Then we restrict $f_i$ to each of these $N_{i,j} \cdot Z(G_0) U_0$, and extend it by $0$ to $(S
  \setminus N_{i,j}) \cdot Z(G_0)U_0$. Denote the function thus obtained by $f_{i,j}$. Again, all $f_{i,j}$ lie in $\Ind^{G_0}_{Z(G_0)U_0}(K_{\hat{\eta},\chi}')^\frd$, and $x \mapsto f_{i,j}(x) \otimes \phi_{i,s(i,j)}(x) = f_{i,j}(x) \otimes v_i$, for a suitably chosen $s(i,j) \in S$. Obviously, we have: $f = \sum_{i,j} f_{i,j} \otimes \phi_{i,s(i,j)}$. Indeed, by construction,
both $f|_S$ and the restriction of the sum to $S$ coincide. And both are functions in $\Ind^{G_0}_{Z(G_0)U_0}(K_{\hat{\eta},\chi}' \otimes_K V)^\frd$; therefore, they are equal. The claim follows.

Now we prove that $\cF^G_{\hat{\eta},\chi}(N_{\Theta,\eta},V)$ is topologically irreducible. Here we follow the same idea as in Theorem \ref{irred_2}. Indeed we need to show that there are no non-trivial morphisms
$$X_{|G_0 \cap tG_0t^{-1}} \to t\cdot X_{|G_0 \cap tG_0t^{-1}}$$
for any non-trivial element $t \in Z(G)G_0\backslash G / G_0.$ Here we even show that there no non-trivial morphisms considered as $U(\frg)$-modules. 
But $X= X(N_{\Theta,\eta},{\bf 1}) \otimes V$ as $U(\frg)$-modules!
Thus the statement follows from what we have showed in Theorem \ref{irred_2}.
\end{proof}

\vskip18pt

\bibliographystyle{plain}
\bibliography{Whittaker}

\end{document}